\documentclass[runningheads]{llncs}
\usepackage{amsmath, amssymb, enumitem, mathtools}
\usepackage[bookmarks, bookmarksnumbered, linktocpage, urlcolor = cyan]{hyperref}

\begin{document}

\title{On the Convergence Rate of the One-Hop Transfer Algorithm}
\author{Ruichao Jiang \and Long Wen}
\authorrunning{R. Jiang and L. Wen}
\institute{Derivation Technology Ltd.\\ \email{\{ruichao.jiang,long.wen\}@derivation.info}}

\maketitle

\begin{abstract}
The transfer algorithm~\cite{jiang} solves the on-chain one-hop swap routing problem. In \cite{jiang}, the convergence is proved but the convergence rate is left open. We prove that the algorithm terminates in at most $\mathcal{O}(N\kappa\log\frac{1}{\varepsilon})$ rounds, where $N$ is the number of AMMs, $\kappa$ a liquidity heterogeneity parameter and $\varepsilon$ a tolerance parameter.
\end{abstract}

\section{Introduction}
Automated Market Maker (AMM) Routing algorithms examine fragmented liquidities across different exchanges for the best swap path \cite{angeris-routing,diamandis-routing,diamandis2024convexnetworkflows,jiang}. Jiang et al.\cite{jiang} proposed the transfer algorithm to achieve on-chain efficiency, which is adopted by Monday Trade protocol \cite{monday}. Xi and Moallemi \cite{xi-moallemi} performed an empirical study to show that this simple algorithm actually outperforms many historical real transaction results.

However, the convergence rate is left open in \cite{jiang}. In this paper, we prove that the one-hop transfer algorithm terminates in $\mathcal{O}\left(N\kappa\log\frac{1}{\varepsilon}\right)$ rounds, where $N$ is the number of pools, $\kappa$ is the liquidity heterogeneity (defined below), and $\varepsilon$ is the price tolerance.

Our experiment shows that our convergence rate is too loose on $\kappa$ and the experiment result suggests that the real rate is almost independent of $\kappa$.

The organization of this article is as follows. In Section \ref{sec:background}, we fix our notation. In Section \ref{sec:main-result}, we prove our result. In Section \ref{sec:experiment}, we present our numerical experiment. We conclude in Section \ref{sec:discussion}.
\section{Background} \label{sec:background}
We recall some notations in \cite{jiang}. Denote by $X/Y$ the token that we sell/buy. An AMM is \textit{trading function}
\begin{equation*}
    E_{x,y}:x_{\text{in}}\mapsto E_{x,y}(x_{\text{in}}),
\end{equation*}
which gives the amount of $Y$ obtained when selling $x_{\text{in}}$ of $X$ to the AMM with inventory $(x,y)$. The trading function is non-negative, monotonically increasing, and concave.

The \textit{post-trade price} of $X$ in $Y$ is $E'(x_{\text{in}})$, which decreases with $x_{\text{in}}$ by concavity. The price of $Y$ in $X$ is $P_i(x_{\text{in}})\coloneqq\frac{1}{E'_i(x_{\text{in}})}$.

The one-hop allocation problem is as follows: Given $x$ amount of token $X$ and $N$ AMMs with trading functions $E_1,\ldots,E_N$,
\begin{align*}
    &\text{maximize} \quad F(\mathbf{x}) = \sum_{i=1}^N E_i(x_i),\\
    &\text{subject to} \quad x_i\geq0,\quad \sum_{i=1}^N x_i = X,
\end{align*}
where $x_i$ is the fraction of input allocated to pool $i$. The optimality condition says that all active pools share the same post-allocation price $E_i'(x_i^*)=\lambda^*$.

Then we briefly describes the transfer algorithm. At each round, the donor $D$ is defined to be the AMM with the highest price of $Y$ in $X$:
\begin{equation*}
    D\coloneqq\arg\max_{i:\,x_i>0} P_i(x_i)
\end{equation*}
and the receiver
\begin{equation*}
    R\coloneqq\arg\min_{i} P'_i(x_i).
\end{equation*}
The transfer algorithm iteratively transfers allocation from $D$ to $R$. A transfer of amount $\delta$ is said to be \textit{legitimate} if after the transfer,
\begin{equation} \label{eqn:legitimacy}
    P_R\leq P_D.
\end{equation}
The algorithm finds a legitimate $\delta$ by halving the transfer amount: Start with $\delta=\frac{x_D}{2}$ and repeatedly halve until legitimacy holds. The algorithm terminates when $\frac{P_D-P_R}{P_D}<\varepsilon$.

We introduce the \textit{smoothness parameter}:
\begin{equation*}
    -E_i''(x)\leq L,
\end{equation*}
i.e. $E_i'$ is $L$-Lipschitz. Since $[0,X]$ is compact, we can always find such $L$.

Then we introduce our only assumption in this paper: AMMs trading functions are assumed to be strictly concave: There exists $\mu>0$ s.t.
\begin{equation*}
    \mu\leq-E_i''(x),
\end{equation*}
for each $i$ and $x\in[0,X]$. $\mu$ is called the \textit{strong concavity parameter}.
Only constant sum AMMs are not strongly concave as its $\mu=0$. Such AMM is in some sense degenerate. In practice, once a constant sum AMM becomes a receiver, we should take as many $Y$ from it as possible because there is no price impact.
\begin{definition}[Liquidity heterogeneity]
    The liquidity heterogeneity $\kappa$ is defined as
    \begin{equation*}
        \kappa=\frac{L}{\mu}.
    \end{equation*}
\end{definition}
\section{Main result} \label{sec:main-result}
Denote by
\begin{equation*}
    h(\textbf{x})\coloneqq F(\mathbf{x}^*)-F(\mathbf{x})\geq0
\end{equation*}
the \textit{objective gap} and by
\begin{equation*}
    g_{\mathbf{x}}\coloneqq E_R'(x_R)-E_D'(x_D)>0
\end{equation*}
the \textit{donor-receiver gap} when the allocation is $\mathbf{x}$.

Furthermore, we define the function
\begin{equation*}
    \varphi(t)\coloneqq E_R'(x_R+t)-E_D'(x_D-t),
\end{equation*}
$t\in[0,x_D]$.

Then
\begin{equation} \label{eqn:varphi-0}
    \varphi(0)=g_{\textbf{x}}>0
\end{equation}
and
\begin{equation} \label{eqn:varphi-derivative}
    \varphi'(t)=E_R''(x_R+t)+E_D''(x_D-t)\in[-2L,-2\mu].
\end{equation}
When referring to a specific round $k$, we use the superscript $(k)$ for each variable. The only exception is $g$, for which we use subscript $g_k$ because we will come across $g$ squared.
\begin{lemma} \label{lemma:legitimacy}
    $\varphi(x_D)<0$.
\end{lemma}
\begin{proof}
    Suppose that the current round is $k$. Suppose for contradiction that $\varphi(x_D)\geq0$, i.e. $E_R'(x_R+x_D)-E_D'(0)\geq0$. Since pool $D$ acquired its allocation in some previous round $k'<k$, in which $D$ was selected as the receiver, it would follow that
    \begin{equation*}
        \max_{i}E'_i\left(x_i^{(k')}\right)=E'_D(0)\leq E_R'(x_R+x_D)<E_R'(x_R)= \max_{i}E'_i\left(x_i^{(k)}\right),
    \end{equation*}
    where the second inequality is by the concavity of $E_R$, contradicting that the receiver's price of $Y$ is non-decreasing in rounds.
\end{proof}
\begin{remark}
   We claimed that the allocation to $D$ is from a round in which $D$ was the receiver. This excludes, for example, the uniform allocation as the initialization method for the transfer algorithm. But it is easily achieved by using a greedy initialization, which includes the simplest allocating all input to the pool with the best initial price.
\end{remark}
Then, we prove a simple geometric lemma.
\begin{lemma}[Sandwich lemma] \label{lemma:sandwich}
    \begin{equation*}
        \max_{i}\left|E'_i(x_i)-\lambda^*\right|\leq g_{\mathbf{x}}.
    \end{equation*}
\end{lemma}
\begin{proof}
    Denote by $I$ the interval $\left[E'_D(x_D),E'_R(x_R)\right]$. By definition of donor and receiver, $E'_i(x_i)\in I$ for all $i$. By the legitimate requirement (Eqn \eqref{eqn:legitimacy}), along the iteration of the transfer algorithm, $I$ form a collection of nested intervals with a limit point $\lambda^*$. Hence, $\lambda^*\in I$. The result follows from the diameter of $I$ is precisely $g_{\textbf{x}}$.
\end{proof}
We show that after each legitimate transfer, the output improvement is not small.
\begin{lemma}[Improvement lower bound] \label{lemma:improvement-bound}
    After each round of the transfer algorithm,
    \begin{equation*}
        F(\mathbf{x}')-F(\mathbf{x})\geq\frac{3g_{\mathbf{x}}^2}{16L},
    \end{equation*}
    where $\mathbf{x}'$ is the allocation vector after the legitimate transfer.
\end{lemma}
\begin{proof}
    Denote by $\delta$ the legitimate transfer amount from $D$ to $R$. By the fundamental theorem of calculus,
    \begin{align*}
        F(\mathbf{x}')-F(\mathbf{x})&=\left[E_R(x_R+\delta)+E_D(x_D-\delta)\right]+\left[E_R(x_R)-E_D(x_D)\right]\\
        &=\left[E_R(x_R+\delta)-E_R(x_R)\right]+\left[E_D(x_D-\delta)-E_D(x_D)\right]\\
        &=\int_0^\delta E_R'(x_R+t)\,dt - \int_0^\delta E_D'(x_D-t)\,dt\\
        &=\int_0^\delta\varphi(t)\,dt.
    \end{align*}
    By the mean value theorem, there exists $t^*\in(0,t)$ s.t.
    \begin{equation*}
        \frac{\varphi(t)-\varphi(0)}{t-0}=\frac{\varphi(t)-g_{\mathbf{x}}}{t}=\varphi'(t^*).
    \end{equation*}
    By Eqn \eqref{eqn:varphi-derivative},
    \begin{equation} \label{eqn:linear-bound}
        \varphi(t)\geq g_{\mathbf{x}}-2Lt.
    \end{equation}
    Since $\varphi(0)>0$ (Eqn \eqref{eqn:varphi-0}) and $\varphi(x_D)<0$ (Lemma \ref{lemma:legitimacy}), by the intermediate value theorem, there exists $\delta^*\in(0,x_D)$ s.t. $\varphi(\delta^*)=0$.
    
    Substituting to Eqn \eqref{eqn:linear-bound},
    \begin{equation*}
        \delta^*\geq\frac{g_{\mathbf{x}}-\varphi(\delta^*)}{2L}=\frac{g_{\mathbf{x}}}{2L}.
    \end{equation*}
    Since the halving procedure finds the largest $\delta$ with $\varphi(\delta)\geq0$ and rejects $2\delta$,
    \begin{equation*}
        \delta\geq\frac{\delta^*}{2}\geq\frac{g_{\mathbf{x}}}{4L}
    \end{equation*}
    Since $g_{\mathbf{x}}-2Lt>0$ for $t\in\left[0,\frac{g_\mathbf{x}}{4L}\right]$,
    \begin{equation*}
        \int_0^\delta\varphi(t)\,dt \geq \int_0^{\frac{g_{\mathbf{x}}}{4L}}\!(g_{\mathbf{x}}-2Lt)\,dt = \frac{g_{\mathbf{x}}^2}{4L}-\frac{g_{\mathbf{x}}^2}{16L}=\frac{3g_{\mathbf{x}}^2}{16L}.
    \end{equation*}
\end{proof}
We prove an estimate on derivative.
\begin{lemma}[Gradient estimate] \label{lemma:gradient-estimate}
    At each round $k$ of the transfer algorithm,
    \begin{equation*}
        g^{(k)}\geq\sqrt{\frac{\mu h^{(k)}}{2N}}.
    \end{equation*}
\end{lemma}
\begin{proof}
    By the concavity of $E_i$,
    \begin{equation*}
        E_i(x_i^*)\leq E_i\left(x^{(k)}_i\right)+E_i'\left(x^{(k)}_i\right)\left(x_i^*-x^{(k)}_i\right).
    \end{equation*}
    Sum over $i$,
    \begin{equation*}
        h^{(k)}\leq\sum_i E_i'\left(x^{(k)}_i\right)\left(x_i^*-x^{(k)}_i\right).
    \end{equation*}
    Subtract a zero term $\lambda^*\sum_i\left(x_i^*-x^{(k)}_i\right)=0$,
    \begin{equation} \label{eqn:concavity}
        \begin{split}
            h^{(k)}&\leq\sum_i\left[E_i'\left(x^{(k)}_i\right)-\lambda^*\right]\left(x_i^*-x^{(k)}_i\right)\\
            &\leq\sum_i\left|E_i'\left(x^{(k)}_i\right)-\lambda^*\right|\cdot\left|x_i^*-x^{(k)}_i\right|\\
            &\leq g^{(k)}\left\|\mathbf{x}^{(k)}-\mathbf{x}^*\right\|_1\\
            &\leq g^{(k)}\sqrt{N}\left\|\mathbf{x}^{(k)}-\mathbf{x}^*\right\|_2
        \end{split}
    \end{equation}
    where the second last inequality is by Lemma \ref{lemma:sandwich} and the last by Cauchy-Schwarz.

    On the other hand, by $\mu$-strong concavity of each $E_i$,
    \begin{align*}
        E_i\left(x^{(k)}_i\right)&\leq E_i(x_i^*)+E'_i(x_i^*)\left(x^{(k)}_i-x_i^*\right)-\frac{\mu}{2}\left(x^{(k)}_i-x_i^*\right)^2\\
        &=E_i(x_i^*)+\lambda^*\left(x^{(k)}_i-x_i^*\right)-\frac{\mu}{2}\left(x^{(k)}_i-x_i^*\right)^2
    \end{align*}
    Sum over $i$, we obtain
    \begin{equation} \label{eqn:strong-concavity}
        h^{(k)}\geq\frac{\mu}{2}\left\|\mathbf{x}^{(k)}-\mathbf{x}^*\right\|^2_2.
    \end{equation}
    Combine Eqns \eqref{eqn:concavity} and \eqref{eqn:strong-concavity},
    \begin{equation*}
        h^{(k)}\leq g^{(k)}\sqrt{\frac{2Nh^{(k)}}{\mu}}\implies g^{(k)}\geq\sqrt{\frac{\mu h^{(k)}}{2N}},
    \end{equation*}
    as desired.
\end{proof}
\begin{theorem}[Convergence rate] \label{thm:main}
    The transfer algorithm satisfies 
    \begin{equation*}
        h^{(k+1)}\leq\left(1-\frac{3}{32\kappa N}\right)^{k+1}\,h^{(0)}.
    \end{equation*}
    The objective gap is at most $\varepsilon$ after $\mathcal{O}\left(N\kappa\log\frac{1}{\varepsilon}\right)$ rounds.
\end{theorem}
\begin{proof}
    By Lemma \ref{lemma:improvement-bound},
    \begin{equation*}
        h^{(k+1)}\leq h^{(k)}-\frac{3g_k^2}{16L}.
    \end{equation*}
    By Lemma \ref{lemma:gradient-estimate},
    \begin{equation*}
        h^{(k+1)}\leq h^{(k)}-\frac{3}{16L}\frac{\mu h^{(k)}}{2N}=\left(1-\frac{3}{32\kappa N}\right)h^{(k)}.
    \end{equation*}
    Iterate and the theorem is proved.
    
    For $h^{(k)}\leq\varepsilon$, we need $k\geq\frac{32N\kappa}{3}\ln\frac{h^{(0)}}{\varepsilon}=\mathcal{O}\left(N\kappa\log\frac{1}{\varepsilon}\right)$.
\end{proof}
\begin{remark}
    Note that here we used $\varepsilon$ as a tolerance term in the final output but $\varepsilon$ was defined in \cite{jiang} as a tolerance term in terms of the relative price difference of the final donor and receiver. This difference is however immaterial to the convergence rate.
\end{remark}
\section{Experiment} \label{sec:experiment}
We demonstrate that the $\kappa$ factor in Theorem \ref{thm:main} is loose via an experiment on $N=10$ constant-product pools. Half the pools have fixed reserves $(100,100)$ and the other half $(100s,100s)$, $s\in\{1,2,5,10,50,100,500,1000\}$.
\begin{table}[htbp!]
    \centering
    \begin{tabular}{c| c| c}
        \hline
        $s$ & $\kappa$ & rounds\\
        \hline
        1   &     8 &  0\\
        2   &     8 & 75\\
        5   &     9 & 70\\
        10  &    13 & 80\\
        50  &    53 & 75\\
        100 &   103 & 85\\
        500 &   503 & 90\\
        1000&  1003 & 70\\
        \hline
    \end{tabular}
    \caption{$\kappa$ and total rounds. $X=100$ and $\varepsilon=10^{-10}$}
    \label{tab:kappa-dependence}
\end{table}

As $\kappa$ increases from $8$ to $1003$, the round count stays between $70$ and $90$, essentially constant.

\section{Discussion} \label{sec:discussion}
In this article, we proved the convergence rate $\mathcal{O}(N\kappa\log\frac{1}{\varepsilon})$ for the one-hop transfer algorithm. The linear dependence on $N$, which essentially comes from the application of Cauchy-Schwarz to bound $\ell_1$ norm by $\ell_2$ norm, happens to be optimal: To determine the best allocation, at least all $N$ pools have to be checked.

However, our numerical experiment suggests a much lesser dependence on $\kappa$: The number of rounds almost doesn't depend on $\kappa$, or at most some logarithmic dependence. We conjecture that during the transfer algorithm, the dependence on $\kappa$ quickly decreases in certain sense. However we are not able to obtain useful quantitative estimate on this.

In practice, large amount of $X$ can trigger many tick crossing in Uniswap v3 type AMMs, which leads to out of gas. This is however not related to the rate of convergence.

\bibliographystyle{splncs04}
\bibliography{biblio}

\end{document}